\documentclass[11pt]{article}

\usepackage{amsmath,amsthm}
\usepackage{amssymb}
\usepackage{hyperref}

\allowdisplaybreaks

\begin{document}

\title{Recurrence relations of poly-Cauchy numbers by the $r$-Stirling transform 
}  
\author{
Takao Komatsu\\
\small Department of Mathematical Sciences, School of Science\\
\small Zhejiang Sci-Tech University\\
\small Hangzhou 310018 China\\
\small \texttt{komatsu@zstu.edu.cn}
}

\date{
}

\maketitle

\def\fl#1{\left\lfloor#1\right\rfloor}
\def\cl#1{\left\lceil#1\right\rceil}
\def\ang#1{\left\langle#1\right\rangle}
\def\stf#1#2{\left[#1\atop#2\right]} 
\def\sts#1#2{\left\{#1\atop#2\right\}}

\newtheorem{theorem}{Theorem}
\newtheorem{Prop}{Proposition}
\newtheorem{Cor}{Corollary}
\newtheorem{Lem}{Lemma}

\begin{abstract}  
We give some formulas of poly-Cauchy numbers by the $r$-Stirling transform. In the case of the classical or poly-Bernoulli numbers, the formulas are with Stirling numbers of the first kind. In our case of the classical or poly-Cauchy numbers, the formulas are with Stirling numbers of the second kind. We also discuss annihilation formulas for poly-Cauchy number with negative indices.    
{\bf MR Subject Classifications:} Primary 11B75; Secondary 11B73, 11B65, 11B37, 11B68, 05A19      
\end{abstract}

\allowdisplaybreaks

\section{Introduction} 

A rich number of results by several authors have been developed concerning the 
combinatorial and analytical properties of a family of numbers and polynomials called poly-Bernoulli and poly-Cauchy numbers and polynomials, respectively. These new sequences have applications in Number Theory (for example, relations with the generalized zeta function and multiple zeta function) and Enumerative Combinatorics (for example, lonesum matrices and set partitions).  
Recently, Ohno and Sasaki \cite{OS1,OS2} establish some formulae (including annihilation formulae) of poly-Bernoulli numbers and polynomials by using $r$-Stirling numbers. We study those of poly-Cauchy numbers and polynomials. On the other hand, expressing Bernoulli numbers or their generalizations in terms of Stirling numbers and harmonic numbers has been an active research.    

In this paper, we mainly focus on the $r$-Stirling transform (analogous to the well-known binomial transform) of the poly-Cauchy numbers of both kinds.  Expressing poly-Cauchy numbers in terms of Stirling numbers and harmonic numbers is also an active research. Some formulae are related to shifted poly-Cauchy numbers.
We can find annihilation formulas for poly-Cauchy numbers with negative indices (analogous to the results given by Ohno and Sasaki for the poly-Bernoulli case). To obtain the results we use some combinatorial properties of the $r$-Stirling numbers of the second kind.

The binomial transform takes the sequence $\{a_n\}_{n\ge 0}$ to the sequence $\{b_n\}_{n\ge 0}$ via the transformation 
$$
b_n=\sum_{k=0}^n\binom{n}{k}a_k
$$ 
with its inverse 
$$
a_n=\sum_{k=0}^n(-1)^{n-k}\binom{n}{k}b_k 
$$ 
(see, e.g., \cite{Gould,Spivey}). For example, for Fibonacci numbers $F_n$, defined by $F_n=F_{n-1}+F_{n-2}$ ($n\ge 3$) with $F_1=F_2=1$, we have $a_n=F_n\rightarrow b_n=F_{2 n}$ and $a_n=(-1)^n F_n\rightarrow b_n=-F_{n}$.   
Similarly, the Stirling transform takes the sequence $\{a_n\}_{n\ge 0}$ to the sequence $\{b_n\}_{n\ge 0}$ via the transformation 
\begin{equation}
b_n=\sum_{k=0}^n\stf{n}{k}a_k
\label{eq:trans1}
\end{equation} 
with its inverse 
\begin{equation}
a_n=\sum_{k=0}^n(-1)^{n-k}\sts{n}{k}b_k 
\label{eq:trans2}
\end{equation}
(see, e.g., \cite{Spivey}). 
Here, the (unsigned) Stirling numbers of the first kind $\stf{n}{k}$ are defined by the coefficients of the rising factorial:
$$
(x)^{(n)}=x(x+1)\cdots(x+n-1)=\sum_{k=0}^n\stf{n}{k}x^k\,. 
$$   
The Stirling numbers of the second kind are expressed as 
$$
\sts{n}{k}=\frac{1}{k!}\sum_{i=0}^k(-1)^i\binom{k}{i}(k-i)^n\,. 
$$ 
For example, by the Stirling transform we have $a_n=\ell^n\rightarrow b_n=(n+\ell-1)!/(\ell-1)!$ ($\ell\ge 1$) and $a_n=(n)^{(l)}\rightarrow b_n=l!\stf{n+1}{l+1}$ ($l\ge 0$).   
Many more examples for binomial transforms and Stirling transforms can be seen in \cite{oeis}.

There are many kinds of generalizations of the Stirling numbers. One of the most interesting ones is the $r$-Stirling number in \cite{Broder}. 
The main purpose of this paper is to yield the relations associated with the inverse and extended transformation of (\ref{eq:trans1}) with (\ref{eq:trans2}). Namely, we consider the transformation by the $r$-Stirling numbers of the second kind: 
\begin{equation}
b_n=\sum_{k=0}^n\sts{n}{k}_r a_k
\label{orth-ri}
\end{equation} 
with its inverse 
\begin{equation}
a_n=\sum_{k=0}^n(-1)^{n-k}\stf{n}{k}_r b_k\,,  
\label{orth-r}
\end{equation}  
that is yielded from the orthogonal relation in (\ref{or-ri}) and (\ref{or-r}). More details concerning the $r$-Stirling numbers, which are used in this paper,  are mentioned in the next section.  
There are some different definitions for the $r$-Stirling numbers of both kinds. They can defined combinatorially as follows.  
The (unsigned) $r$-Stirling numbers of the first kind $\stf{n}{m}_r$ are defined as the number of permutations of the set $\{1,\dots,n\}$ having $m$ cycles, such that the numbers $1,\dots,r$ are in distinct cycles. The $r$-Stirling numbers of the second kind $\sts{n}{m}_r$ are defined as the number of ways to partition the set $\{1,\dots,n\}$ into $m$ non-empty disjoint subsets such that the numbers $1,\dots,r$ are in different subsets.

Bernoulli number $B_n$ with its various generalizations is one of the most interesting and most popular numbers in history. For every integer $k$, the poly-Bernoulli numbers $\mathbb B_n^{(k)}$ are defined by 
$$
\frac{{\rm Li}_k(1-e^{-x})}{1-e^{-x}}=\sum_{n=0}^\infty\mathbb B_n^{(k)}\frac{x^n}{n!}\,,
$$ 
where 
$$
{\rm Li}_k(z)=\sum_{n=1}^\infty\frac{z^n}{n^k}
$$ 
is the polylogarithm function.  When $k=1$, $\mathbb B_n=\mathbb B_n^{(1)}$ is the Bernoulli number, defined by 
$$
\frac{x e^x}{e^x-1}=\sum_{n=0}^\infty\mathbb B_n\frac{x^n}{n!}
$$ 
with $\mathbb B_1=1/2$.  

In \cite{OS1,OS2}, the annihilation formulas are obtained for positive integers $\mu$ and $k$ with $\mu\ge k$:  
$$
\sum_{j=0}^\mu(-1)^j\stf{\mu+2}{j+2}_2\mathbb B_k^{(n-j)}=0
$$
or 
$$
\sum_{j=0}^\mu(-1)^j\stf{\mu+2}{j+2}_2\mathbb B_{n+j}^{(-k)}=0\,. 
$$ 
On the other hand, similar formulas with positive index have been studied
 in \cite{KCDC,Rahmani}. Surprisingly, relations include the harmonic numbers $H_n:=\sum_{\ell=1}^n 1/\ell$ and their generalization. For example, we have 
$$
\sum_{j=1}^n\stf{n}{j}\mathbb B_j^{(k)}=\frac{n!}{(n+1)^k}\quad(n\ge 1)  
$$ 
and 
$$
\sum_{j=0}^n\stf{n+1}{j+1}\mathbb B_j^{(k)}=n! H_{n+1}^{(k)}\quad(n\ge 0)\,,  
$$ 
where $H_n^{(k)}=\sum_{i=1}^n(1/i^k)$ is the higher order harmonic number.  
Some other simpler cases are given by   
\begin{equation}
\sum_{j=0}^n(-1)^j\stf{n+1}{j+1}B_j=n! H_{n+1}\quad(n\ge 1) 
\label{eq:ce}
\end{equation} 
(\cite{CE}),  
and 
$$
\sum_{j=1}^n\stf{n}{j}B_j=-\frac{(n-1)!}{n+1}\quad(n\ge 1) 
$$  
(\cite{Sprugnoli}), 
where the Bernoulli number $B_n$ is defined by 
$$
\frac{x}{e^x-1}=\sum_{n=0}^\infty B_n\frac{x^n}{n!}
$$ 
with $B_1=-1/2$. In \cite[Corollary 10]{CD} and \cite[Theorem 5]{KC}, Bernoulli polynomials are expressed in terms of Stirling numbers and hyperharmonic numbers, which is a generalization of harmonic numbers. Thus, the formula (\ref{eq:ce}) can be obtained as a special case. Apart from that formula, other expressions can be found in \cite{Boyadzhiev}.

As the inverse function of $e^x-1$ is $\log(1+x)$, Cauchy number $c_n$ is also enough interesting in history.  
In \cite{Ko1,Ko2}, for any integer $k$, the poly-Cauchy numbers $c_n^{(k)}$ are defined by 
$$
{\rm Lif}_k\bigl(\log(1+x)\bigr)=\sum_{n=0}^\infty c_n^{(k)}\frac{x^n}{n!}\,, 
$$ 
where 
$$
{\rm Lif}_k(z)=\sum_{n=0}^\infty\frac{z^n}{n!(n+1)^k}
$$ 
is the polylogarithm factorial (polyfactorial) function.  
The poly-Cauchy numbers $c_n^{(k)}$ have an explicit expression in terms of the Stirling numbers of the first kind as 
\begin{equation}  
c_n^{(k)}=\sum_{\ell=1}^n\frac{(-1)^{n-\ell}}{(\ell+1)^k}\stf{n}{\ell}\quad(n\ge 1)\,. 
\label{pc:st1} 
\end{equation}
When $k=1$, $c_n=c_n^{(1)}$ is the classical Cauchy number (e.g., see, \cite{Comtet}), defined by 
$$
\frac{x}{\log(1+x)}=\sum_{n=0}^\infty c_n\frac{x^n}{n!}\,.  
$$ 
Note that $b_n=c_n/n!$ are call the Bernoulli numbers of the second kind (e.g., see, \cite{Howard}).  Cauchy numbers have many similar or corresponding relations to Bernoulli numbers.  

In this paper, we show that for integers $n$, $r$ and $k$ with $n\ge r\ge 1$, 
$$ 
\sum_{j=r}^n\sts{n}{j}_r c_j^{(k)}=\sum_{\ell=1}^{r}\frac{(-1)^{r-\ell}}{(n-r+\ell+1)^k}\stf{r}{\ell}
$$ 
and for integers $n$, $r$ and $k$ with $n\ge r-1\ge 0$, 
$$ 
\sum_{j=r-1}^n\sts{n+1}{j+1}_r c_j^{(k)}=\sum_{\ell=1}^{r}(-1)^{r-\ell}\stf{r}{\ell}\sum_{i=0}^{n-r+\ell}\binom{n-r+\ell}{i}\frac{1}{(i+1)^k}\,. 
$$ 
In the case of negative indices, we have for $n\ge k+2$,  
$$ 
\sum_{l=0}^k\sts{n-1}{n-l-1}_{n-k-1}c_{n-l}^{(-k)}=0\,. 
$$ 
We also give their analogous results for the poly-Cauchy numbers of the second kind $\widehat c_n^{(k)}$, defined in (\ref{def:pc2}).

\section{The $r$-Stirling numbers}  

In this section, we mention some basic properties of the $r$-Stirling numbers, which will be used in this paper.     

By using the $r$-Stirling numbers, 
the original Stirling numbers can be expressed as 
$$
\stf{n}{m}=\stf{n}{m}_0,\quad \sts{n}{m}=\stf{n}{m}_0\,,
$$
and 
$$
\stf{n}{m}=\stf{n}{m}_1,\quad \sts{n}{m}=\stf{n}{m}_1\quad(n>0)\,,
$$ 
There exist recurrence relations as 
\begin{equation}
\stf{n}{m}_r=(n-1)\stf{n-1}{m}_r+\stf{n-1}{m-1}_r\quad(n>r)
\label{rec-rel} 
\end{equation} 
with 
$$
\stf{n}{m}_r=0\quad(n<r)\quad\hbox{and}\quad \stf{n}{m}_r=\delta_{m,r}\quad(n=r)\,, 
$$  
and 
$$
\sts{n}{m}_r=m\sts{n-1}{m}_r+\sts{n-1}{m-1}_r\quad(n>r)
$$
with 
$$
\sts{n}{m}_r=0\quad(n<r)\quad\hbox{and}\quad \sts{n}{m}_r=\delta_{m,r}\quad(n=r)\,, 
$$ 
where  $\delta_{m,r}=1$ ($m=r$); $0$ ($m\ne r$). 
Note that the orthogonal properties \cite[Theorem 5, Theorem 6]{Broder} hold between two kinds of $r$-Stirling numbers too. Namely, 
\begin{equation}
\sum_{k}(-1)^{n-k}\stf{n}{k}_r\sts{k}{m}_r=\begin{cases}
\delta_{m,n}&\text{if $n\ge r$};\\ 
0&\text{otherwise}
\end{cases}
\label{or-ri}
\end{equation}
and 
\begin{equation}
\sum_{k}(-1)^{n-k}\stf{k}{n}_r\sts{m}{k}_r=\begin{cases}
\delta_{m,n}&\text{if $n\ge r$};\\ 
0&\text{otherwise}\,.
\end{cases}
\label{or-r}
\end{equation}
The ordinary generating functions of both kinds of $r$-Stirling numbers \cite[Corollary 9, Corollary 10]{Broder} are given by 
$$
\sum_{k}\stf{n}{k}_r z^k=\begin{cases}
z^r(z+r)(z+r+1)\cdots(z+n-1)&\text{if $n\ge r\ge 0$};\\ 
0&\text{otherwise}  
\end{cases}
$$
and 
$$
\sum_{k}\sts{k}{m}_r z^k=\begin{cases}
\dfrac{z^m}{(1-r z)\bigl(1-(r+1)z\bigr)\cdots\bigl(1-m z\bigr)}&\text{if $m\ge r\ge 0$};\\ 
0&\text{otherwise}\,, 
\end{cases}
$$
respectively.  

\section{Main results}  

This is our first main result.  

\begin{theorem}  
For integers $n$, $r$ and $k$ with $n\ge r\ge 1$, we have 
$$ 
\sum_{j=r}^n\sts{n}{j}_r c_j^{(k)}=\sum_{\ell=1}^{r}\frac{(-1)^{r-\ell}}{(n-r+\ell+1)^k}\stf{r}{\ell}\,. 
$$ 
\label{th1}
\end{theorem}  

\noindent 
{\it Remark.}  
When $r=1,2,3,4$, we have 
\begin{align*} 
\sum_{j=1}^n\sts{n}{j}c_j^{(k)}&=\frac{1}{(n+1)^k}\quad\text{(\cite[Theorem 3]{Ko1})}\,,\\ 
\sum_{j=2}^n\sts{n}{j}_2 c_j^{(k)}&=-\frac{1}{n^k}+\frac{1}{(n+1)^k}\,,\\ 
\sum_{j=3}^n\sts{n}{j}_3 c_j^{(k)}&=\frac{2}{(n-1)^k}-\frac{3}{n^k}+\frac{1}{(n+1)^k}\,,\\ 
\sum_{j=4}^n\sts{n}{j}_4 c_j^{(k)}&=-\frac{6}{(n-2)^k}+\frac{11}{(n-1)^k}-\frac{6}{n^k}+\frac{1}{(n+1)^k}\,. 
\end{align*}
When $n=r$, by $\sts{r}{r}_r=1$, we have (\ref{pc:st1}) with $n=r$. 

One variation is given as follows.  

\begin{theorem}  
For integers $n$, $r$ and $k$ with $n\ge r-1\ge 0$, we have 
$$ 
\sum_{j=r-1}^n\sts{n+1}{j+1}_r c_j^{(k)}=\sum_{\ell=1}^{r}(-1)^{r-\ell}\stf{r}{\ell}\sum_{i=0}^{n-r+\ell}\binom{n-r+\ell}{i}\frac{1}{(i+1)^k}\,. 
$$ 
\label{th2}
\end{theorem}  

\noindent 
{\it Remark.}  
When $r=1$, we have 
$$
\sum_{j=0}^n\sts{n+1}{j+1}c_j^{(k)}=\sum_{i=0}^{n}\binom{n}{i}\frac{1}{(i+1)^k}\,. 
$$
When $k=1$ and $r=1,2,3$, we have 
\begin{align*} 
\sum_{j=0}^n\sts{n+1}{j+1}c_j&=\frac{2^{n+1}-1}{n+1}\,,\\ 
\sum_{j=0}^n\sts{n+1}{j+1}_2 c_j&=-\frac{2^{n}-1}{n}+\frac{2^{n+1}-1}{n+1}\,,\\ 
\sum_{j=0}^n\sts{n+1}{j+1}_3 c_j&=\frac{2(2^{n-1}-1)}{n-1}-\frac{3(2^{n}-1)}{n}+\frac{2^{n+1}-1}{n+1}\,. 
\end{align*}
When $n=r-1$, we have 
\begin{align*}  
c_{r-1}^{(k)}&=\sum_{\ell=0}^{r-1}(-1)^{r-\ell}\stf{r}{\ell}\sum_{i=0}^{\ell-1}\binom{\ell-1}{i}\frac{1}{(i+1)^k}\\
&=\sum_{i=0}^{r-1}\frac{1}{(i+1)^k}\sum_{\ell=i+1}^r(-1)^{r-\ell}\stf{r}{\ell}\binom{\ell-1}{i}\\
&=\sum_{i=0}^{r-1}\frac{(-1)^{r-i-1}}{(i+1)^k}\stf{r-1}{i}\,, 
\end{align*}  
yielding the formula in (\ref{pc:st1}) with $n=r-1$.

Poly-Cauchy numbers of the second kind $\widehat c_n^{(k)}$ (\cite{Ko1}) are defined by  
\begin{equation}
{\rm Lif}_k\bigl(-\log(1+x)\bigr)=\sum_{n=0}^\infty\widehat c_n^{(k)}\frac{x^n}{n!}\,.
\label{def:pc2}
\end{equation}  
When $k=1$, $\widehat c_n=\widehat c_n^{(1)}$ are the classical Cauchy numbers of the second kind (see, e.g., \cite{Comtet}), defined by 
$$
\frac{x}{(1+x)\log(1+x)}=\sum_{n=0}^\infty\widehat c_n^{(k)}\frac{x^n}{n!}\,.
$$ 
Concerning the poly-Cauchy numbers of the second kind $\widehat c_n^{(k)}$, we have the following corresponding results.  

\begin{theorem}  
For integers $n$, $r$ and $k$ with $n\ge r\ge 1$, we have 
$$ 
\sum_{j=r}^n\sts{n}{j}_r\widehat c_j^{(k)}=\sum_{\ell=1}^{r}\frac{(-1)^{n}}{(n-r+\ell+1)^k}\stf{r}{\ell}\,. 
$$ 
\label{th11}
\end{theorem} 

\begin{theorem}  
For integers $n$, $r$ and $k$ with $n\ge r-1\ge 0$, we have 
$$ 
\sum_{j=r-1}^n\sts{n+1}{j+1}_r\widehat c_j^{(k)}=\sum_{\ell=1}^{r}(-1)^{r-\ell}\stf{r}{\ell}\sum_{i=0}^{n-r+\ell}\binom{n-r+\ell}{i}\frac{(-1)^i}{(i+1)^k}\,.
$$ 
\label{th22}
\end{theorem}

\section{Proof of the main results}  

We need the following relations in order to prove Theorem \ref{th1}. 

\begin{Lem}  
Let $r$ be any integer with $1\le r\le n$. Then for $1\le m\le n-r+1$,  
$$ 
\sum_{\ell=1}^m\stf{r}{\ell}\stf{n}{r-\ell+m}_r=\stf{n}{m}\,,
$$
and for $n-r+2\le m\le n$,  
$$ 
\sum_{\ell=1}^{n+1-\max\{m,r\}}\stf{r}{m-n+r-1+\ell}\stf{n}{n-\ell+1}_r=\stf{n}{m}\,. 
$$ 
\label{lem11}
\end{Lem}  

\noindent 
{\it Remark.} 
There are some different ways to express the Stirling numbers of the first kind in terms of the $r$-Stirling numbers of the first kind by choosing different $r$. For example,  
\begin{align*}
\stf{6}{3}=225&=\stf{2}{1}\stf{6}{4}_2+\stf{2}{2}\stf{6}{3}_2\\
&=1\cdot 71+1\cdot 154\\
&=\stf{3}{1}\stf{6}{5}_3+\stf{3}{2}\stf{6}{4}_3+\stf{3}{3}\stf{6}{3}_3\\
&=2\cdot 12+3\cdot 47+1\cdot 60\\
&=\stf{4}{1}\stf{6}{6}_4+\stf{4}{2}\stf{6}{5}_4+\stf{4}{3}\stf{6}{4}_4\\
&=6\cdot 1+11\cdot 9+6\cdot 20\,. 
\end{align*}
\begin{proof}
According to the definition of the $r$-Stirling numbers of the first kind, the relation 
$$
\stf{6}{3}=\stf{4}{1}\stf{6}{6}_4+\stf{4}{2}\stf{6}{5}_4+\stf{4}{3}\stf{6}{4}_4
$$ 
can be explained as follows.  The permutations of the set $\{1,2,3,4,5,6\}$ with $3$ cycles has three different ways, by fixing the first numbers $1,2,3,4$.  
\begin{enumerate} 
\item[(1)] $6$ numbers are divided into 6 different cycles (so that $1,2,3,4$ are in distinct cycles). Then, the cycles including $1,2,3,4$ are collected into the same group (permutation). 
$$
(1)(2)(3)(4)(\ldots)(\ldots)\Longrightarrow \bigl((1)(2)(3)(4)\bigr)\bigl(\ldots\bigr)\bigl(\ldots\bigr)
$$ 
\item[(2)] $6$ numbers are divided into 5 different cycles so that $1,2,3,4$ are in distinct cycles. Then, 4 cycles including $1,2,3,4$ are collected into 2 different groups (permutation). 
\begin{align*}
(1\ldots)(2\ldots)(3\ldots)(4\ldots)(\ldots)&\Longrightarrow \bigl((a\ldots)(b\ldots)\bigr)\bigl((c\ldots)(d\ldots)\bigr)\bigl(\ldots\bigr)\\ 
\text{or}&\Longrightarrow \bigl((a\ldots)(b\ldots)(c\ldots)\bigr)\bigl(d\ldots\bigr)\bigl(\ldots\bigr)\,, 
\end{align*}
where $\{a,b,c,d\}=\{1,2,3,4\}$
\item[(3)] $6$ numbers are divided into 4 different cycles so that $1,2,3,4$ are in distinct cycles. Then, 4 cycles including $1,2,3,4$ are collected into 3 different groups (permutation). 
$$
(1\ldots)(2\ldots)(3\ldots)(4\ldots)\Longrightarrow \bigl((a\ldots)(b\ldots)\bigr)\bigl(c\ldots\bigr)\bigl(d\ldots\bigr)
$$ 
where $\{a,b,c,d\}=\{1,2,3,4\}$
\end{enumerate}
In general, the permutations of the set $\{1,2,\dots,n\}$ with $m$ cycles are done by the following way. When $1\le m\le n-r+1$, $n$ numbers are divided into $r-\ell+m$ different cycles so that $1,2,\dots,r$ are in distinct cycles. Then, the $r$ cycles including $1,2,\dots,r$ are collected into $\ell$ different groups. Here, $1\le \ell\le m$.  
\begin{multline*}
(1\ldots)(2\ldots)\cdots(r\ldots)\underbrace{(\ldots)\cdots(\ldots)}_{m-\ell}\\ 
\Longrightarrow 
\underbrace{\bigl((a_1\ldots)\ldots\bigr)\cdots\bigl(\ldots(a_r\ldots)\bigr)}_{\ell}\underbrace{\bigl(\ldots\bigr)\cdots\bigl(\ldots\bigr)}_{m-\ell}\,,
\end{multline*}  
where $\{a_1,\dots,a_r\}=\{1,\dots,r\}$. 
When $n-r+2\le m\le n$, by putting $m'=m-n+r-1$, $n$ numbers are divided into $n-\ell+m=r+(m-m'-\ell)$ different cycles so that $1,2,\dots,r$ are in distinct cycles. Then, the $r$ cycles including $1,2,\dots,r$ are collected into $m'+\ell$ different groups. Here, $1\le \ell\le n+1-\max\{m,r\}$.  
\begin{multline*}
(1\ldots)(2\ldots)\cdots(r\ldots)\underbrace{(\ldots)\cdots(\ldots)}_{m-m'-\ell}\\ 
\Longrightarrow 
\underbrace{\bigl((a_1\ldots)\ldots\bigr)\cdots\bigl(\ldots(a_r\ldots)\bigr)}_{m'+\ell}\underbrace{\bigl(\ldots\bigr)\cdots\bigl(\ldots\bigr)}_{m-m'-\ell}\,,
\end{multline*}  
where $\{a_1,\dots,a_r\}=\{1,\dots,r\}$. 
\end{proof}

Now, it is ready to prove Theorem \ref{th1}.  

\begin{proof}[Proof of Theorem \ref{th1}]  
By the transformation of the $r$-Stirling numbers in (\ref{orth-ri}) with (\ref{orth-r}), equivalently, we shall prove that 
\begin{align}
c_n^{(k)}&=\sum_{i=r}^n(-1)^{n-i}\stf{n}{i}_r\sum_{\ell=1}^r\frac{(-1)^{r-\ell}}{(i-r+\ell+1)^k}\stf{r}{\ell}\notag\\ 
&=\sum_{\ell=1}^r(-1)^{r-\ell}\stf{r}{\ell}\sum_{i=r}^n\frac{(-1)^{n-i}}{(i-r+\ell+1)^k}\stf{n}{i}_r\quad(n\ge r)\,.
\label{eq111} 
\end{align}
By Lemma \ref{lem11}, setting $m=i-r+\ell$, the right-hand side of (\ref{eq111}) is equal to 
\begin{align*}
&\sum_{m=1}^{n-r+1}\frac{(-1)^{n-m}}{(m+1)^k}\sum_{\ell=1}^m\stf{r}{\ell}\stf{n}{r-\ell+m}_r\\
&\qquad +\sum_{m=n-r+2}^{n}\frac{(-1)^{n-m}}{(m+1)^k}\sum_{\ell=1}^{n+1-\max\{m,r\}}\stf{r}{m-n+r-1+\ell}\stf{n}{n-\ell+1}_r\\
&=\sum_{m=1}^n\frac{(-1)^{n-m}}{(m+1)^k}\stf{n}{m}=c_n^{(k)}\quad\text{(\ref{pc:st1})}\,. 
\end{align*}
\end{proof}

\begin{proof}[Proof of Theorem \ref{th2}] 
As similar orthogonal relations to (\ref{or-ri}) and (\ref{or-r}), we have 
$$ 
\sum_{k}(-1)^{n-k}\stf{n+1}{k+1}_r\sts{k+1}{m+1}_r=\begin{cases}
\delta_{m,n}&\text{if $n\ge r-1$};\\ 
0&\text{otherwise}
\end{cases}
$$ 
and 
$$ 
\sum_{k}(-1)^{n-k}\stf{k+1}{n+1}_r\sts{m+1}{k+1}_r=\begin{cases}
\delta_{m,n}&\text{if $n\ge r-1$};\\ 
0&\text{otherwise}\,. 
\end{cases}
$$ 
Thus, equivalently, we shall prove that 
\begin{align*}  
c_n^{(k)}&=\sum_{j=r-1}^n(-1)^{n-j}\stf{n+1}{j+1}_{r}\sum_{\ell=1}^r(-1)^{r-\ell}\stf{r}{\ell}\sum_{i=0}^{j-r+\ell}\binom{j-r+\ell}{i}\frac{1}{(i+1)^k}\\
&=\sum_{\ell=1}^r\stf{r}{\ell}\sum_{m=\ell-1}^{n+\ell-r}(-1)^{n-m}\stf{n+1}{m-\ell+r+1}_r\sum_{i=0}^m\binom{m}{i}\frac{1}{(i+1)^m}\,.
\end{align*}  
Because of the expression in (\ref{pc:st1}), it is sufficient to prove that for $n\ge r$
\begin{equation} 
\stf{n}{i}=\sum_{\ell=1}^r\stf{r}{\ell}\sum_{j=1}^{n-r+2}(-1)^{\ell+j-i}\binom{\ell+j-2}{i}\stf{n+1}{r+j-1}_r\,. 
\label{eq:303}
\end{equation}
Put for $0\le i\le n$ and $n\ge 1$ 
$$
a_{n,i}:=\sum_{\ell=1}^r\stf{r}{\ell}\sum_{j=1}^{n-r+2}(-1)^{\ell+j-i}\binom{\ell+j-2}{i}\stf{n+1}{r+j-1}_r\,.
$$ 
First, for $i=0$ and $r\ge 2$, by 
$$
\sum_{\ell=1}^r(-1)^{\ell}\stf{r}{\ell}=0\,, 
$$  
we have 
\begin{align*}
a_{n,0}&=\sum_{\ell=1}^r\stf{r}{\ell}\sum_{j=1}^{n-r+2}(-1)^{\ell+j}\binom{\ell+j-2}{0}\stf{n+1}{r+j-1}_r\\
&=\sum_{j=1}^{n-r+2}(-1)^{j}\stf{n+1}{r+j-1}_r\sum_{\ell=1}^r(-1)^{\ell}\stf{r}{\ell}\\
&=0\,. 
\end{align*} 
For $i=0$ and $r=1$,  
$$ 
a_{n,0}=\stf{1}{1}\sum_{j=1}^{n+1}(-1)^{j+1}\stf{n+1}{j}=0\quad(n\ge 1)\,. 
$$ 
Next, let $i\ge 1$. For convenience, put  
$$ 
a_{n,i}=\sum_{\ell=1}^r\stf{r}{\ell}b_{n,i,\ell}
$$ 
with for fixed $\ell$ 
$$
b_{n,i}=b_{n,i,\ell}:=\sum_{j=1}^{n-r+2}(-1)^{\ell+j-i}\binom{\ell+j-2}{i}\stf{n+1}{r+j-1}_r\,.
$$ 
Then, using the recurrence relation (\ref{rec-rel}), 
by 
$$
\stf{n}{r+j-1}_r=0\quad(j=0,\,j=n-r+2)  
$$ 
and 
$$
\binom{\ell+j-1}{i}=\binom{\ell+j-2}{i}+\binom{\ell+j-2}{i-1}\,, 
$$ 
we have 
\begin{align*}
b_{n,i}&=n\sum_{j=1}^{n-r+2}(-1)^{\ell+j-i}\binom{\ell+j-2}{i}\stf{n}{r+j-1}_r\\
&\quad +\sum_{j=1}^{n-r+2}(-1)^{\ell+j-i}\binom{\ell+j-2}{i}\stf{n}{r+j-2}_r\\
&=n\sum_{j=1}^{n-r+1}(-1)^{\ell+j-i}\binom{\ell+j-2}{i}\stf{n}{r+j-1}_r\\
&\quad -\sum_{j=0}^{n-r+1}(-1)^{\ell+j-i}\binom{\ell+j-1}{i}\stf{n}{r+j-1}_r\\
&=(n-1)\sum_{j=1}^{n-r+1}(-1)^{\ell+j-i}\binom{\ell+j-2}{i}\stf{n}{r+j-1}_r\\
&\quad -\sum_{j=0}^{n-r+1}(-1)^{\ell+j-i}\binom{\ell+j-2}{i-1}\stf{n}{r+j-1}_r\\
&=(n-1)b_{n-1,i}+b_{n-1,i-1}\,. 
\end{align*}
Hence, $a_{n,i}=(n-1)a_{n-1,i}+a_{n-1,i-1}$\quad($i\ge 1$).  
Since $a_{n,i}$ and $\stf{n}{i}$ satisfy the same recurrence relation as (\ref{rec-rel}) with $a_{n,0}=\stf{n}{0}$\quad($n\ge 1$), the relation (\ref{eq:303}) holds.  
\end{proof}

The proofs of Theorem \ref{th11} and Theorem \ref{th22} are similar to those of Theorem \ref{th1} and Theorem \ref{th2}, respectively, and omitted.

\section{Poly-Cauchy polynomials}  

Poly-Cauchy polynomials $c_n^{(k)}(z)$ are defined by 
\begin{equation}  
(1+x)^z{\rm Lif}_k\bigl(\log(1+x)\bigr)=\sum_{n=0}^\infty c_n^{(k)}(z)\frac{x^n}{n!}
\label{def:pcp}
\end{equation}
(\cite[Theorem 2]{Kamano}),  
and an explicit expression is given by 
\begin{equation}  
c_n^{(k)}(z)=\sum_{m=0}^n\stf{n}{m}(-1)^{n-m}\sum_{i=0}^m\binom{m}{i}\frac{z^{m-i}}{(i+1)^k}\,. 
\label{exp:pcp}
\end{equation}
(\cite[Theorem 1]{Kamano}).  When $z=0$, $c_n^{(k)}=c_n^{(k)}(0)$ are the poly-Cauchy numbers. The definition in \cite{Ko2} is an alternative way, simply by replacing $z$ by $-z$. As an extension of Theorem \ref{th1} and Theorem \ref{th2}, we have the following.  

\begin{theorem}  
For integers $n$, $r$ and $k$ with $n\ge r\ge 1$ and a real number $q$, we have 
$$ 
\sum_{j=r}^n\sts{n}{j}_r c_j^{(k)}(q)=\sum_{\ell=1}^{r}(-1)^{r-\ell}\stf{r}{\ell}\sum_{i=0}^{n-r+\ell}\binom{n-r+\ell}{i}\frac{q^{n-i}}{(i+1)^k}\,. 
$$ 
\label{th5}
\end{theorem}  

\noindent 
{\it Remark.}  
When $k=1$, we have the relation of the Cauchy polynomials $c_j(z)=c_n^{(1)}(z)$: 
$$
\sum_{j=r}^n\sts{n}{j}_r c_j(q)=\sum_{\ell=1}^{r}(-1)^{r-\ell}\stf{r}{\ell}\frac{(q+1)^{n-r+\ell+1}-q^{n-r+\ell+1}}{n-r+\ell+1}\,.
$$ 
When $q=0$ in Theorem \ref{th5}, Theorem \ref{th1} is reduced.  
When $q=1$ in Theorem \ref{th5}, we find  
$$
\sum_{j=r}^n\sts{n}{j}_r c_j^{(k)}(1)=\sum_{j=r-1}^n\sts{n+1}{j+1}_r c_j^{(k)}\,.
$$  

\begin{proof}[Proof of Theorem \ref{th5}]  
The proof is similar to that of Theorem \ref{th2}. We prove that 
\begin{align*}  
c_n^{(k)}(q)&=\sum_{j=r-1}^n(-1)^{n-j}\stf{n+1}{j+1}_{r}\sum_{\ell=1}^r(-1)^{r-\ell}\stf{r}{\ell}\sum_{i=0}^{j-r+\ell}\binom{j-r+\ell}{i}\frac{q^{j-r+\ell-i}}{(i+1)^k}\\
&=\sum_{\ell=1}^r\stf{r}{\ell}\sum_{m=\ell-1}^{n+\ell-r}(-1)^{n-m}\stf{n+1}{m-\ell+r+1}_r\sum_{i=0}^m\binom{m}{i}\frac{q^{m-i}}{(i+1)^m}\,.
\end{align*}  
Then, it is also sufficient to prove (\ref{eq:303}).  
\end{proof}

Poly-Cauchy polynomials of the second kind $\widehat c_n^{(k)}(z)$ are defined by 
\begin{equation}  
\frac{{\rm Lif}_k\bigl(-\log(1+x)\bigr)}{(1+x)^z}=\sum_{n=0}^\infty\widehat c_n^{(k)}(z)\frac{x^n}{n!}
\label{def:pcp}
\end{equation}
(\cite[Theorem 5]{Kamano}),  
and an explit expression is given by 
\begin{equation}  
\widehat c_n^{(k)}(z)=\sum_{m=0}^n\stf{n}{m}(-1)^{n}\sum_{i=0}^m\binom{m}{i}\frac{z^{m-i}}{(i+1)^k}\,. 
\label{exp:pcp}
\end{equation}
(\cite[Theorem 4]{Kamano}).  When $z=0$, $\widehat c_n^{(k)}=\widehat c_n^{(k)}(0)$ are the poly-Cauchy numbers of the second kind. As an extension of Theorem \ref{th11}, we have the following.  The proof is similar and omitted.  

\begin{theorem}  
For integers $n$, $r$ and $k$ with $n\ge r\ge 1$ and a real number $q$, we have 
$$ 
\sum_{j=r}^n\sts{n}{j}_r\widehat c_j^{(k)}(q)=\sum_{\ell=0}^{r-1}(-1)^{r-\ell}\stf{r}{\ell}\sum_{i=0}^{n-r+\ell}\binom{n-r+\ell}{i}\frac{(-q)^{n-r-\ell}q^{-i}}{(i+1)^k}\,.
$$ 
\label{th55}
\end{theorem}  

\noindent 
Notice that when $q=1$, 
$$
\sum_{j=r}^n\sts{n}{j}_r\widehat c_j^{(k)}(1)\ne\sum_{j=r-1}^n\sts{n+1}{j+1}_r\widehat c_j^{(k)}\,.
$$

\section{Shifted poly-Cauchy numbers and harmonic numbers}  

The right-hand sides of Theorem \ref{th1} and Theorem \ref{th11} can be written in terms of shifted poly-Cauchy numbers of both kinds, defined in \cite{KS}.  

For integers $n\ge 0$ and $k\ge 1$ and a real number $\alpha>0$, the shifted poly-Cauchy numbers of the first kind $c_{n,\alpha}^{(k)}$ are defined by 
$$
c_{n,\alpha}^{(k)}=\underbrace{\int_0^1\cdots\int_0^1}_k(x_1\cdots x_k)^\alpha(x_1\cdots x_k-1)\cdots(x_1\cdots x_k-n+1)d x_1\dots d x_k
$$ 
and can be expressed in terms of the Stirling numbers of the first kind as 
$$
c_{n,\alpha}^{(k)}=\sum_{m=0}^n\stf{n}{m}\frac{(-1)^{n-m}}{(m+\alpha)^k}\quad(n\ge 0,~k\ge 1)
$$ 
(\cite[Theorem 1]{KS}).  When $\alpha=1$, $c_{n}^{(k)}=c_{n,1}^{(k)}$ are the original poly-Cauchy numbers of the first kind.  
Similarly, for integers $n\ge 0$ and $k\ge 1$ and a real number $\alpha>0$, the shifted poly-Cauchy numbers of the second kind $\widehat c_{n,\alpha}^{(k)}$ are defined by 
$$
\widehat c_{n,\alpha}^{(k)}=\underbrace{\int_0^1\cdots\int_0^1}_k(-x_1\cdots x_k)^\alpha(-x_1\cdots x_k-1)\cdots(-x_1\cdots x_k-n+1)d x_1\dots d x_k
$$ 
and can be expressed in terms of the Stirling numbers of the first kind as 
$$
\widehat c_{n,\alpha}^{(k)}=(-1)^n\sum_{m=0}^n\stf{n}{m}\frac{1}{(m+\alpha)^k}\quad(n\ge 0,~k\ge 1)
$$ 
(\cite[Proposition 2]{KS}).  When $\alpha=1$, $\widehat c_{n}^{(k)}=\widehat c_{n,1}^{(k)}$ are the original poly-Cauchy numbers of the second kind.  

By using the shifted poly-Cauchy numbers of the both kinds, the identities in Theorem \ref{th1} and Theorem \ref{th11} can be written in terms of shifted poly-Cauchy numbers of both kinds.  

\begin{Cor}  
For integers $n$, $r$ and $k$ with $n\ge r\ge 1$, we have 
\begin{align*}
\sum_{j=r}^n\sts{n}{j}_r c_j^{(k)}&=c_{r,n-r+1}^{(k)}\,,\\   
\sum_{j=r}^n\sts{n}{j}_r\widehat c_j^{(k)}&=\widehat c_{r,n-r+1}^{(k)}\,. 
\end{align*}  
\label{cor:th1-3}
\end{Cor} 
 
In addition, it would be noticed that a similar formula in Theorem \ref{th1} is given as a special case in \cite[Corollary 8]{Kargin}.   
\bigskip

Latter identities in Theorem \ref{th22}, Theorem \ref{th5} and Theorem \ref{th55} are relate to the higher order harmonic numbers.  
For a sequence $t=(t_1,t_2,\dots)$, the Bell polynomials $\Omega_i(t):=\Omega_i(t_1,t_2,\dots,t_i)$ are defined by 
$$
\sum_{i=0}^\infty\Omega_i(t)\frac{x^i}{i!}=\exp\left(\sum_{k=1}^\infty t_k\frac{x^k}{k!}\right) 
$$ 
and expressed as 
$$
\Omega_i(t)=\sum_{a_1+2 a_2+\cdots+i a_i=i\atop a_1,a_2,\dots,a_i\ge 0}\frac{i!}{a_1!a_2!\cdots a_i!}\left(\frac{t_1}{1}\right)^{a_1}\left(\frac{t_2}{2}\right)^{a_1}\cdots\left(\frac{t_i}{i}\right)^{a_i}\,. 
$$  
It is known and has been studied by several people that 
\begin{align}  
&i! m\binom{m+n}{n}\sum_{k=0}^n(-1)^k\binom{n}{k}\frac{1}{(m+k)^{i+1}}\notag\\ 
&=\Omega_i(H_{m+n}-H_{m-1},H_{m+n}^{(2)}-H_{m-1}^{(2)},\dots,H_{m+n}^{(i)}-H_{m-1}^{(i)})
\label{eq:wj}
\end{align}  
(see, e.g., \cite[(3.56)]{WJ} and references therein). By using this general harmonic number identity (\ref{eq:wj}), Theorem \ref{th22}, the case $q=-1$ of Theorem \ref{th5}, and the cases $q=-1$ of Theorem \ref{th55} can be written in terms of the higher order harmonic number numbers.     
\begin{Cor}  
For integers $n$, $r$ and $k$ with $n\ge r-1\ge 0$, we have 
$$
\sum_{j=r-1}^n\sts{n+1}{j+1}_r\widehat c_j^{(k)}=\sum_{\ell=1}^{r}(-1)^{r-\ell}\stf{r}{\ell}\frac{\Omega_{k-1}(H_{n-r+\ell+1},H_{n-r+\ell+1}^{(2)},\dots,H_{n-r+\ell+1}^{(k-1)})}{(n-r+\ell+1)(k-1)!}\,. 
$$
For integers $n$, $r$ and $k$ with $n\ge r\ge 1$, we have 
\begin{align*}  
\sum_{j=r}^n\sts{n}{j}_r c_j^{(k)}(-1)&=\sum_{\ell=1}^{r}(-1)^{n-r+\ell}\stf{r}{\ell}\frac{\Omega_{k-1}(H_{n-r+\ell+1},H_{n-r+\ell+1}^{(2)},\dots,H_{n-r+\ell+1}^{(k-1)})}{(n-r+\ell+1)(k-1)!}\,,\\  
\sum_{j=r}^n\sts{n}{j}_r\widehat c_j^{(k)}(-1)&=\sum_{\ell=0}^{r-1}(-1)^{r-\ell}\stf{r}{\ell}\frac{\Omega_{k-1}(H_{n-r+\ell+1},H_{n-r+\ell+1}^{(2)},\dots,H_{n-r+\ell+1}^{(k-1)})}{(n-r+\ell+1)(k-1)!}\,. 
\end{align*}
\label{cor:th4-5-6} 
\end{Cor}

\section{Annihilation formulas for poly-Cauchy numbers}  

Some annihilation formulas for poly-Bernoulli numbers with negative indices have been established in \cite{OS1,OS2}.  In this section, we show annihilation formulas for poly-Cauchy numbers with negative indices.  In \cite{Ko3}, some annihilation formulas for poly-Cauchy numbers have been done, but the expressions are not so elegant. With the aid of the $r$-Stirling numbers, we can give more elegant forms.

\begin{theorem}
For $n\ge k+2$,  
$$ 
\sum_{l=0}^k\sts{n-1}{n-l-1}_{n-k-1}c_{n-l}^{(-k)}=0\,. 
$$ 
\label{th31}  
\end{theorem} 

\noindent 
{\it Remark.}  
For $k=1,2,3$ we have 
\begin{align*}  
0&=\sts{n-1}{n-1}_{n-2}c_n^{(-1)}+\sts{n-1}{n-2}_{n-2}c_{n-1}^{(-1)}\\
&=c_n^{(-1)}+(n-2)c_{n-1}^{(-1)}\quad(n\ge 3)\,,\\ 
0&=\sts{n-1}{n-1}_{n-3}c_n^{(-2)}+\sts{n-1}{n-2}_{n-3}c_{n-1}^{(-2)}+\sts{n-1}{n-3}_{n-3}c_{n-2}^{(-2)}\\
&=c_n^{(-2)}+(2 n-5)c_{n-1}^{(-2)}+(n-3)^2 c_{n-2}^{(-2)}\quad(n\ge 4)\,,\\
0&=\sts{n-1}{n-1}_{n-4}c_n^{(-3)}+\sts{n-1}{n-2}_{n-4}c_{n-1}^{(-3)}\\
&\qquad+\sts{n-1}{n-3}_{n-4}c_{n-2}^{(-3)}+\sts{n-1}{n-4}_{n-4}c_{n-3}^{(-3)}\\
&=c_n^{(-3)}+(3 n-9)c_{n-1}^{(-3)}+(3 n^2-21 n+37)c_{n-2}^{(-3)}+(n-4)^3 c_{n-3}^{(-3)}\quad(n\ge 5)\,. 
\end{align*}  

\begin{proof}[Proof of Theorem \ref{th31}] 
Since the $r$-Stirling numbers of the second kind can be expressed as 
$$
\sts{n+m}{n}_r=\sum_{r\le i_1\le \cdots\le i_m\le n}i_1 i_2\cdots i_m
$$ 
(\cite[Theorem 8]{Broder}), together with \cite[Theorem 2.1]{Ko3}, we have 
\begin{align*}
0&=\sum_{l=0}^k\left(\sum_{l+1\le i_1\le\cdots\le i_l\le k+1}(n-i_1)\cdots(n-i_l)\right)c_{n-l}^{(-k)}\\ 
&=\sum_{l=0}^k\sts{n-1}{n-l-1}_{n-k-1}c_{n-l}^{(-k)}\,. 
\end{align*}
\end{proof}

Similarly, concerning the poly-Cauchy numbers of the second kind $\widehat c_n^{(-k)}$, we have the following annihilation formula.  

\begin{theorem}
For $n\ge k+2$,  
$$ 
\sum_{l=0}^k\sts{n+1}{n-l+1}_{n-k}\widehat c_{n-l}^{(-k)}=0\,. 
$$ 
\label{th32}  
\end{theorem} 

\noindent 
{\it Remark.}  
For $k=0,1,2$ we have 
\begin{align*}  
0&=\sts{n+1}{n+1}_{n}\widehat c_n^{(0)}+\sts{n+1}{n}_{n}\widehat c_{n-1}^{(0)}\\
&=\widehat c_n^{(0)}+n\widehat c_{n-1}^{(0)}\quad(n\ge 1)\,,\\ 
0&=\sts{n+1}{n+1}_{n-1}\widehat c_n^{(-1)}+\sts{n+1}{n}_{n-1}\widehat c_{n-1}^{(-1)}+\sts{n+1}{n-1}_{n-1}\widehat c_{n-2}^{(-1)}\\
&=\widehat c_n^{(-1)}+(2 n-1)\widehat c_{n-1}^{(-1)}+(n-1)^2\widehat c_{n-2}^{(-1)}\quad(n\ge 2)\,,\\
0&=\sts{n+1}{n+1}_{n-2}\widehat c_n^{(-2)}+\sts{n+1}{n}_{n-2}\widehat c_{n-1}^{(-2)}\\
&\qquad+\sts{n+1}{n-1}_{n-2}\widehat c_{n-2}^{(-2)}+\sts{n+1}{n-2}_{n-2}\widehat c_{n-3}^{(-2)}\\
&=\widehat c_n^{(-2)}+(3 n-3)\widehat c_{n-1}^{(-2)}+(3 n^2-9 n+7)\widehat c_{n-2}^{(-2)}+(n-2)^3\widehat c_{n-3}^{(-2)}\quad(n\ge 3)\,. 
\end{align*}  

\begin{proof}[Proof of Theorem \ref{th32}] 
By using \cite[Theorem 8]{Broder}) again, together with \cite[Theorem 3.1]{Ko3}, we have 
\begin{align*}
0&=\sum_{l=0}^{k+1}\left(\sum_{l-1\le i_1\le\cdots\le i_l\le k}(n-i_1)\cdots(n-i_l)\right)\widehat c_{n-l}^{(-k)}\\ 
&=\sum_{l=0}^k\sts{n+1}{n-l+1}_{n-k}\widehat c_{n-l}^{(-k)}\,. 
\end{align*}
\end{proof}



\end{document}